\documentclass{amsart}
\usepackage{amsmath,amsthm}
\usepackage{amsfonts,amssymb}
\usepackage{accents}
\usepackage{enumerate}
\usepackage{accents,color}
\usepackage{graphicx}
\usepackage{wrapfig}
\newcommand{\lvt}{\left|\kern-1.35pt\left|\kern-1.3pt\left|}
\newcommand{\rvt}{\right|\kern-1.3pt\right|\kern-1.35pt\right|}

\makeatletter
\def\@tocline#1#2#3#4#5#6#7{\relax
  \ifnum #1>\c@tocdepth 
  \else
    \par \addpenalty\@secpenalty\addvspace{#2}%
        \begingroup \hyphenpenalty\@M
    \@ifempty{#4}{%
      \@tempdima\csname r@tocindent\number#1\endcsname\relax
    }{%
      \@tempdima#4\relax
    }%
    \parindent\z@ \leftskip#3\relax \advance\leftskip\@tempdima\relax
    \rightskip\@pnumwidth plus4em \parfillskip-\@pnumwidth
    #5\leavevmode\hskip-\@tempdima #6\nobreak\relax
    \ifnum #1=0
    \hfil\hbox to\@pnumwidth{}
    \else
    \hfil\hbox to\@pnumwidth{\@tocpagenum{#7}}\fi
    \par
    \nobreak
    \endgroup
  \fi}
\makeatother

\newtheorem{thm}{Theorem}[section]

\newtheorem{lem}[thm]{Lemma}
\newtheorem{prop}[thm]{Proposition}

\newtheorem{defn}[thm]{Definition}
 
\theoremstyle{remark}

 \def\la{{\langle}}
 \def\ra{{\rangle}}
 \def\ve{{\varepsilon}}
\def\({\left(}
\def \){ \right)}
\def\[{\left[}
\def \]{ \right]}

 \def\d{\mathrm{d}}
 
 \def\sph{{\mathbb{S}^{d-1}}}

 \def\sE{{\mathsf E}}

 \def\sO{{\mathsf O}}
 \def\sP{{\mathsf P}}

 \def\sc{{\mathsf c}}
 \def\sd{{\mathsf d}}
 \def\sm{{\mathsf m}}
 \def\sw{{\mathsf w}}

 \def\a{{\alpha}}
 \def\b{{\beta}}
 \def\g{{\gamma}}
 \def\k{{\kappa}}
 \def\t{{\theta}}
 \def\l{{\lambda}}
 
 \def\s{\sigma}
 \def\la{{\langle}}
 \def\ra{{\rangle}}
 \def\ve{{\varepsilon}}

 \def\CH{{\mathcal H}}

 \def\CV{{\mathcal V}}

 \def\BB{{\mathbb B}}

 \def\RR{{\mathbb R}}
 \def\SS{{\mathbb S}}
 
 \def\VV{{\mathbb V}}
 \def\XX{{\mathbb X}}
 
      \def\proj{\operatorname{proj}}

\def\lla{\langle{\kern-2.5pt}\langle}      
\def\rra{\rangle{\kern-2.5pt}\rangle}

\newcommand{\wh}{\widehat}

\def\f{\frac}

\graphicspath{{./}}

\begin{document}
\title{Highly localized kernels on space of homogeneous type}  \author{Yuan~Xu}
\address{Department of Mathematics, University of Oregon, Eugene, 
OR 97403--1222, USA}
\email{yuan@uoregon.edu} 
\thanks{The author was partially supported by Simons Foundation Grant \#849676.}
\thanks{For the proceedings of the {\it Constructive Theory of Functions, Lozenets, Bulgaria, 2023}.}
\date{\today}  

\begin{abstract}  
Highly localized kernels based on orthogonal polynomials have been studied and utilized over several regular 
domains. Much of the results deduced via these kernels can be treated uniformly in the framework of localizable
spaces of homogeneous type. We clarify the concept of localness and provide a list of such localizable spaces. 
\end{abstract}

\maketitle
 
\section{Introduction}
\setcounter{equation}{0}

Highly localized kernels based on orthogonal polynomials have been instrumental in studying various problems
in approximation theory and computational analysis over regular domains in the past two decades. Such kernels 
have near-exponential decay off the diagonal and also serve as kernels for integral operators that are polynomials
of near-best approximation on the domain, which makes them ideal tools for studying several problems in 
approximation theory and harmonic analysis. They have been established for weighted function spaces on the 
interval, the unit sphere, the unit ball, and the simplex in the Euclidean spaces 
\cite{BD, KPX1, KPX2, IPX, IPX2, NPW1, NPW2, PX1, PX2} as well as on $\RR_+^d$ and $\RR^d$
 \cite{KPPX, PX3}, and more recently, on the conic domains \cite{X21, X23a}. 

In the literature, much of the analysis based on highly localized kernels is carried out domain by domain, relying on 
the geometry of each domain. Recently, in \cite{X21}, a framework was developed that provides a unified treatment
for all compact domains that are equipped with highly localized kernels. In this expository paper, we aim to clarify 
the concept of localization and explain the essence of the framework. 

The setup is on the space of homogeneous type, which is a measure space $(\Omega, \mu, \sd)$ on the set $\Omega$ 
with a positive measure $\mu$ and a metric $\sd$ such that all open balls 
$$
    B(x,r)=\{y \in \Omega: \sd(x,y) < r\}  
$$ 
are measurable and $\mu$ is a regular measure satisfying the doubling property
$$
   \mu (B(x,2r)) \le c \mu (B(x,r)), \qquad \forall x \in  \Omega, \quad \forall r >0,
$$ 
where $c$ is independent of $x$ and $r$. Such a measure $\mu$ is called a doubling measure. In our setup, the
domain $\Omega$ will be either an algebraic surface or a domain with a non-empty interior, and 
$\d \mu = \sw(x) \d x$ with $\sw$ being a nonnegative weight function defined on $\Omega$ that defines a 
well-established inner product
\begin{equation} \label{eq:ipd}
    \la f, g\ra_{\sw} = \int_{\Omega} f(x) g(x) \sw(x) \d x
\end{equation}
on the space of polynomials restricted to $\Omega$. We shall write the homogeneous space as $(\Omega, \sw, \sd)$
with $\sw$ being a doubling weight, and we write $\sw (E) = \int_E \d \sw$ for the subset $E \subset \Omega$. 

We need kernels defined in terms of orthogonal polynomials. Let $\CV_n^d(\sw)$ be the space of orthogonal 
polynomials of degree $n$ with respect to this inner product. Let $P_n(\sw; \cdot,\cdot)$ be the reproducing kernel 
of the space $\CV_n^d(\sw)$. The orthogonal projection operator $\proj_n: L^2(\Omega, \sw)\mapsto \CV_n^d(\sw)$
can be written as 
\begin{equation} \label{eq:proj}
  (\proj_n f)(x) =\int_\Omega P_n(\sw; x,y)f(y)\sw(y) \d y, \qquad f\in L^2(\Omega, \sw).
\end{equation}
Let $\wh a$ be a cut-off function, defined as a compactly supported nonnegative function in $C^\infty(\RR_+)$ such that 
$\wh a(t) = 1$ for $0 \le t \le 1$ and $\wh a(t) = 0$ for $t \ge 2$. Then our highly localized kernels are defined by 
\begin{equation} \label{eq:Ln-kernel} 
L_n(\sw;x,y) := \sum_{j=0}^\infty \wh a \Big(\frac{j}{n}\Big) P_j(\sw; x,y).
\end{equation}
The kernel is regarded as highly localized if it satisfies an upper-bound estimate that contains a factor 
$(1+  n \sd (x,y))^{-\k}$ for any given $\k > 0$, which means that the kernel decays faster than any polynomial 
power away from the diagonal $x=y$; see the next section for a precise definition.

We call $(\Omega,\sw,\sd)$ a {\it localizable} space of homogeneous type if it possesses highly localized 
kernels. For such a space, the framework in \cite{X21} allows us to establish several results in approximation
theory and harmonic analysis in weighted function spaces, such as $L^p(\Omega, \sw)$, that include a 
characterization of best approximation by polynomials and localized tight frames, among other results. In other words, 
several results can be regarded as properties of localizable space of homogeneous type. The definition of 
highly localized kernels in \cite{X21} requires three assertions that need to hold. As we shall show in the next section,
one of the assertions is redundant as it holds for every doubling weight. 

At the moment, the localizable spaces of homogenous type are known only for a few regular domains. This is
not surprising given the complexity of multivariate orthogonal polynomials and, even more so, the kernel
$P_n(\sw;  \cdot,\cdot)$. Most highly localized kernels are established when a closed-form formula for the 
reproducing kernel $P_n(\sw; \cdot, \cdot)$ is known, which we call an {\it addition formula}, taking a cure
from the addition formula for the spherical harmonics. The existence of such a formula is rare and reflects 
something fundamental about the localizable space of homogeneous type. We shall provide a list of examples
of localizable spaces of homogeneous type.  

The paper is organized as follows. In the next section, we provide a precise definition of highly localized 
kernels and, as a consequence, localizable space of homogeneous type, and a sketch of some results that 
can be established with the help of such kernels. In the third section, we provide a list of examples of known
localizable spaces. 

\section{Localized space of homogeneous type}
\setcounter{equation}{0} 

\subsection{Orthogonal structure and localized kernels}
Let $(\Omega,\sw, \sd)$ be a space of homogeneous type associated with the inner product $\la \cdot,\cdot\ra_\sw$
defined in \eqref{eq:ipd}. The dimension of the space $\CV_n^d(\sw)$ of orthogonal polynomials depends on the 
geometry of $\Omega \subset \RR^d$. For our purpose, we only consider two cases. The first case is when the interior 
of $\Omega$ is an open set, such as the unit ball or the simplex in $\RR^d$, for which 
\begin{equation} \label{eq:dimVn}
   \dim \CV_n^d(\Omega,\sw) = \binom{n+d-1}{n}, \quad n = 0, 1,2,\ldots.
\end{equation}
The second case is when the domain $\Omega$ is a quadratic algebraic surface of the form $\{x \in \RR^d: \phi(x) = 0\}$,
where $\phi$ is a quadratic polynomial, such as the unit sphere $\sph$ with $\phi(x) = 1-\|x\|^2$, for which 
\begin{equation} \label{eq:dimVnS}
  \dim \CV_n^d(\Omega,\sw) = \binom{n+d-2}{n} + \binom{n+d-3}{n-1}, \quad n = 1,2,3,\ldots. 
\end{equation}

We assume that $\sw$ is regular so that the orthogonal decomposition 
$$
   L^2(\Omega, \varpi) = \bigoplus_{n=0}^\infty \CV_n^d(\Omega,\sw): \qquad f = \sum_{n=0}^\infty \proj_n f
$$
holds, where $\proj_n$ is the orthogonal projection operator. This gives the Fourier orthogonal decomposition of $f$ since, 
if $P_{\nu,n}$ is an orthonormal basis of $\CV_n^d(\varpi)$, then 
$$
  \proj_n f = \sum_{1 \le \nu \le \dim \CV_n^d(\Omega, \sw)} \wh f_{\nu,n} P_{\nu,n}, \qquad \wh f_{\nu,n} := \la f, P_{\nu,n} \ra_\sw. 
$$
The reproducing kernel $P_n(\sw; \cdot,\cdot)$ of $\CV_n^d(\Omega,\sw)$ is the kernel of the operator $\proj_n$,
as shown in \eqref{eq:proj}, which can be written as 
$$
   P_n(\sw; x,y) =   \sum_{1\le \nu \le \dim \CV_n^d(\Omega, \sw)} P_{\nu,n}(x) P_{\nu,n}(y)
$$
in terms of any orthonormal basis of $\CV_n^d(\varpi)$. In the multidimensional setting, orthogonal polynomials are 
complicated and described by multiple indices even when they can be written down, and the kernel $P_n(\sw; \cdot,\cdot)$
is a sum of multiple layers of products of polynomials. Such a sum can hardly be useful for obtaining estimates 
of kernels like $L_n(\sw)$ defined in \eqref{eq:Ln-kernel}. 

In all cases when highly localized kernels are known, their desired upper bound is established with the help of 
an addition formula for the reproducing kernel. In most of the cases, the addition formula is of the form
\begin{equation} \label{eq:AF}
   P_n(\sw; x,y) = \int_{\RR^m} Z_n(\xi(x,y; u)) \d \mu(u)
\end{equation}
where $Z_n$ is a polynomial, often orthogonal polynomial, of one variable, $\xi: (x,y;z) \mapsto \RR$ for 
$x, y \in \RR^d$, and $\d \mu$ is a measure on $\RR^m$, which can degenerate to point evaluations. The 
simplest example is the unit sphere with orthogonal polynomials being spherical harmonics with respect to the 
surface measure, for which the closed-form formula is given by 
\begin{equation} \label{eq:AF_sphere}
   P_n(\sw;x,y) = \frac{n+\l}{\l} C_n^\l (\la \xi,\eta \ra) \quad \hbox{with} \quad \l = \frac{d-2}{2},
\end{equation}
where $C_n^\l$ is the Gegenbauer polynomial of degree $n$ (see Subsection \ref{sec:sphere} below). This 
formula is called the {\it addition formula} for spherical harmonics, a name that we adopt for closed-formula of 
reproducing kernels in all other cases. If an addition formula of the form \eqref{eq:AF} exists, then the kernel 
$L_n(\sw;\cdot,\cdot)$ defined in \eqref{eq:Ln-kernel} can be written as 
$$
    L_n(\sw; x,y) = \int_{\RR^m} L_n\big(\xi (x,y; \mu)\big) \d \mu(u), 
$$
where $L_n$ on the right-hand side is a function of one variable given by
$$
 L_n(t) =  \sum_{j=0}^\infty \wh a \Big(\frac{j}{n}\Big) Z_j(t), 
$$
which allows us to deduce an upper bound of the kernel $L_n(\sw; \cdot,\cdot)$ from that of $L_n$. 

We should emphasize that the addition formula is something special. If it is of the form \eqref{eq:AF}, then it
indicates that the reproducing kernel, hence the projection operator $\proj_n f$, has an implicit one-dimensional 
structure, which allows us to relate the study of the Fourier orthogonal series to that of the Fourier orthogonal series
in one variable. For example, it is well-known that a substantial portion of the Fourier series on the unit sphere 
is reduced to that of the Fourier-Gegenbauer series. The same holds for other regular domains and for any 
space of homogeneous type that possesses an addition formula (cf. \cite{X21, X22, X23a}). 

\subsection{Highly localized kernels} 

Let $(\Omega, \sw, \sd)$ be a space of homogeneous type. The existing examples of the highly localized
kernel suggests the following definition. 

\begin{defn}\label{def:Assertions}
The kernels $L_n(\sw; \cdot,\cdot)$, $n=1,2,\ldots$, are called highly localized if they satisfy the following 
assertions: 
\begin{enumerate}[\,  \bf 1]
\item[] \textbf{Assertion 1}. For $\k > 0$ and $x, y\in \Omega$, 
$$
   |L_n(\sw; x ,y)| \le c_\k \frac{1} {\sqrt{\sw\!\left(B(x,n^{-1})\right)} \sqrt{\sw\!\left(B(y,n^{-1})\right)}
      \left(1+n \sd(x,y)\right)^\k}.
$$
\medskip\noindent
\item[] 
\textbf{Assertion 2}.  For $0 < \delta \le \delta_0$ with some $\delta_0<1$ and $x_1 \in B(x_2, \frac{\delta}{n})$, 
$$
   \left|L_n(\sw; x_1,y) - L_n(\sw; x_2,y)\right| \le c_\k \frac{n \sd(x_1,x_2)} 
        {\sqrt{\sw\!\left(B(x_1,n^{-1})\right)} \sqrt{\sw\!\left(B(x_2,n^{-1})\right)} \left(1+n \sd(x_2,y)\right)^\k}.
$$
\end{enumerate}
\end{defn} 

This definition contains one less assertion in comparison with the definition given in \cite{X21}. The third 
assertion contained in \cite{X21} implies that the first two assertions are sharp. For example, it implies that 
\begin{equation}\label{eq:Ln-bdd}
   \int_{\Omega} \left|L_n (\sw;x,y) \right|^p \sw(y) \d \sm(y)
       \le c \left[ \sw\!\left(B\!\left(x,n^{-1}\right)\right)\right]^{1-p}, 
\end{equation}
where $\d \sm$ is the Lebesgue measure on $\Omega$, which is frequently needed when we utilize 
the highly localized kernels. The third assertion imposes a condition on the weight function $\sw$: 

\medskip
\textbf{\textit{Assertion 3}}. For sufficient large $\k >0$, there is a constant $c_\k > 0$ such that 
\begin{align*}
\int_{\Omega} \frac{ \sw(y)}{\sw\!\left(B(y,n^{-1})\right)
    \big(1 + n \sd(x,y) \big)^{\k}}    \d \sm(y) \le c_\k.
\end{align*}

This assertion is stated in \cite{X21} as part of the definition for $L_n(\sw)$ being a highly localized kernel. 
It is, however, redundant since it holds for all doubling weight. 

\begin{prop}
If $(\Omega, \sw, \sd)$ is a space of homogeneous type, then Assertion 3 holds. 
\end{prop}

\begin{proof}
It is known that the doubling weight $\sw$ satisfies 
\begin{equation}\label{eq:w(x)/w(y)}
 \big(1+ n \sd(x,y) \big)^{-\a(\sw)} \le 
  \frac{  \sw\!\left(B\left(y,n^{-1}\right)\right)}{ \sw\!\left(B\left(x,n^{-1}\right)\right)}\le \big(1+ n \sd(x,y) \big)^{\a(\sw)},
\end{equation}
where $\a(\sw) > 0$ is a positive constant. Fix $n$ and define $\Omega_{0,n}(x) = B(x, n^{-1})$ and 
$$
\Omega_{m,n}(x) = \left \{y:  2^{m-1}n^{-1} \le \sd (x,y) \le 2^{m} n^{-1}\right\}, \qquad m=1,2,\ldots. 
$$ 
Then $\Omega = \cup_{m=0}^\infty \Omega_{m,n}(x)$. Let $J_n$ denote the left-hand side of the inequality in Assertion 3. 
Using \eqref{eq:w(x)/w(y)}, we obtain
\begin{align*}
  J_n \, & = \sum_{m=0}^\infty \int_{\Omega_{m,n}(x)} \frac{ \sw(y)}{\sw\!\left(B(y,n^{-1})\right) \big(1 + n \sd(x,y) \big)^{\k}} \d \sm(y) \\
 & \le \sum_{m=0}^\infty \frac1{\sw\!\left(B(x, n^{-1})\right)} \int_{\Omega_{m,n}(x)} 
      \frac{ \sw(y)}{\big(1 + n \sd(x,y) \big)^{\k - \a(w)}} \d \sm(y). 
\end{align*}
Since $\sd(x,y) \ge 2^{m-1} n^{-1}$ and $\Omega_{m,n}(x) = B(x,2^m n^{-1}) \setminus B(x,2^{m-1} n^{-1})$, it follows that
\begin{align*}
  J_n \le \sum_{m=0}^\infty \frac1{\sw\!\left(B(x, n^{-1})\right)} \frac{ \sw\!\left(B(x, 2^m n^{-1})\right)} {(2^{m-1})^{\k - \a(\sw)}}  
  \le \sum_{m=0}^\infty \frac{ (c_\sw)^m}{(1+2^{m-1})^{\k - \a(\sw)} } \le \frac{1}{ 2^{ \k - \a(\sw)} - c_\sw}
\end{align*}
where we have used the doubling condition of $\sw$, with $c_\sw$ denoting the smallest constant for the doubling condition, 
in the second inequality. Choosing $\k$ large shows that $J_n$ is bounded.
\end{proof}

\begin{defn}
A space $(\Omega, \sw, \sd)$ of homogeneous type is called localizable if it possesses highly localized kernels for 
one doubling weight $\varpi$ on $\Omega$.
\end{defn}

This definition is given in \cite{X21}, where several problems are studied in a localizable space of homogeneous type,
with few additional assumptions, and it provides a uniform treatment of analysis on several regular domains that
are derived domain by domain in the literature. As a consequence, as soon as we can show a space of homogeneous
type is localizable, a variety of results will follow right away. This is the motivation for the study in \cite{X21}, where the 
framework is developed and applied to new spaces over conic domains. 

It is worthwhile to emphasize that we need highly localized kernel only for {\it one} doubling weight $\varpi$ on 
$\Omega$ in order for $(\Omega, \sw, \sd)$ being localizable. Indeed, most of the results in \cite{X21, X23a} hold
for all doubling weight in a localizable space.

\subsection{Analysis in localizable space of homogeneous type}
The highly localized kernels have been an important tool for several problems in approximation theory and harmonic
analysis. Here we briefly outline a few of them to illustrate their usefulness. 

\subsubsection{Best approximation by polynomials}
Let $L_n(\sw; \cdot,\cdot)$ be the operator defined in \eqref{eq:Ln-kernel} via a cut-off function $\wh a$, which
is a resampling of the kernel $P_{2n}(\sw; \cdot,\cdot)$. We can define an integral operator that has 
$L_n\big(\varpi; \cdot,\cdot)$ as its kernel. For convenience, we denote the operator by $L_n(\varpi)*f$ in the 
notation of a pseudo convolution and define it as 
\begin{equation}\label{eq:Lnf}
   L_n(\sw)* f (x) =  \int_{\Omega} f(y) L_n\big(\sw; x, y) \sw(y)  \d \sm(y). 
\end{equation}
The definition of the cut-off function $\wh a$ shows that $L_n*f = f$, whenever $f$ is a polynomial of degree at most 
$n$, and $L_n*f $ is a polynomial of degree at most $2n$. If $L_n(\sw)$ were highly localized, then $L_n(\sw)*f$
should be a good approximation to the function $f$. This is indeed the case. Let $E_n(f)_{p,\sw}$ denote the 
error of the best polynomial approximation of degree at most $n$ to $f$, 
$$
  E_n(f)_{p,\sw} := \int_{P \in \Pi_n(\Omega)} \|f - P \|_{p,\sw}, \qquad 1 \le p \le \infty,
$$
where $\|\cdot\|_{p,\sw}$ denotes the norma of $L^p(\Omega, \sw)$ and $\Pi_n(\Omega)$ denote the space of
polynomials restricted on $\Omega$. Then $L_n(\sw)*f$ is a near-best polynomial 
of best approximation, as specified in the following theorem.

\begin{thm}
Let $(\Omega,\sw,\sd)$ be a localizable space of homogeneous type. Let $f \in L^p(\Omega, \sw)$, $1 \le p < \infty$ 
and $f \in C(\Omega)$ for $p =\infty$. Then the operator $L_n(\sw)*f$ is bounded in $L^p(\Omega, \sw)$ and 
$$
    \|f - L_n* f\|_{p,\sw} \le c_p  E_n(f)_{p, \sw}, \quad 1 \le p \le \infty.
$$
\end{thm}

If $Z_n$ in the addition formula \eqref{eq:AF} is a normalized Jacobi polynomial given by
$$
          Z_n^{(\a,\b)}(t) = \frac{P_n^{(\a,\b)}(1) P_n^{(\a,\b)}(t)}{h_n^{(\a,\b)}},
$$
where $h_n^{(\a,\b)}$ is the $L^2$ norm of $P_n^{(\a,\b)}$, which holds for several domains as will be
seen in the next section, then we can define a translation operator $S_{\t, \sw}$ by 
  \begin{equation}\label{eq:Stheta}
    \proj_n (\sw; S_{\t, \sw} f) =\frac{P_n^{(\a,\b)}(\cos \t)}{P_n^{(\a,\b)}(1)} \proj_n(\sw; f),  \quad n = 0,1,2,\ldots.
\end{equation}
This is a bounded operator and satisfies, in particular, that
$$
  \|S_{\t, \varpi} f \|_{p,\varpi} \le \|f\|_{p, \varpi} \quad \hbox{and} \quad \lim_{\t\to 0} \|S_{\t, \varpi}  f - f\|_{\varpi,p} =0.
$$
for $f\in L^p(\varpi; \Omega)$ if $1 \le p \le \infty$, or $f \in C(\Omega)$ if $p =\infty$. These properties allow us to
define a modulus of smoothness in the localizable space of homogeneous type. For $r > 0$, we define the $r$-th 
difference operator 
$$
  \triangle_{\t,\varpi}^r = \left(I - S_{\t,\varpi}\right)^{r/2}  =\sum_{n=0}^\infty (-1)^n \binom{r/2}{n} (S_{\t,\varpi})^n,
$$
where $I$ denotes the identity operator, in the distribution sense. Then the modulus smoothness is defined by
\begin{equation}\label{moduli1}
\omega_r(f,t)_{p,\varpi}:= \sup_{0< \theta \le t} \left\|\triangle_{\t,\varpi}^r f\right\|_{p,\varpi}, \quad  0 < t <\pi,
\end{equation}
which satisfies the usual properties of the modulus of smoothness. Furthermore, it can be used to prove both
direct and inverse theorems in a characterization of the best approximation $E_n(f)_{p,\sw}$ by polynomials. 
We refer to \cite{X21} for details. 

\subsubsection{Positive cubature rules and Christoffel functions} 
Let $\sw$ be a doubling weight on $\Omega$. A cubature rule of degree $n$ for $\sw$ is a finite linear
combination of point evaluations that satisfies 
$$
  \int_\Omega f(x) \sw(x)  \d \sm(x) = \sum_{k=1}^N \l_k f(x_k), \qquad  \forall f \in \Pi_n(\Omega),
$$
where $\l_k \in \RR$ and $x_k \in \Omega$. The cubature rule is called positive if all $\l_k$ are positive. 
The positivity is important for numerical stability and for many applications. Positive cubature rules exist
if the set of its nodes is well separated as defined below. 

\begin{defn}\label{defn:separated-pts}
Let $\Xi$ be a discrete set in $\Omega$. 
\begin{enumerate} [  \quad (a)]
\item A finite collection of subsets $\{S_z: z \in \Xi\}$ is called a partition of $\Omega$ if $S_z^\circ\cap 
S_y^\circ  = \emptyset$ when $z \ne y$ and $\Omega = \bigcup_{z \in \Xi} S_z$. 
\item Let $\ve>0$. A discrete subset $\Xi$ of $\Omega$ is called $\ve$-separated if $\sd(x,y) \ge\ve$
for every two distinct points $x, y \in \Xi$. 
\item $\Xi$ is called maximal if there is a constant $c_d > 1$ such that 
\begin{equation*}
  1 \le  \sum_{z\in \Xi} \chi_{B(z, \ve)}(x) \le c_d, \qquad \forall x \in \Omega,
\end{equation*}
where $\chi_E$ denotes the characteristic function of the set $E$.
\end{enumerate}
\end{defn} 

The item (c) implies $\Omega = \bigcup_{z \in \Xi} B(z,\ve)$. The positive cubature rules have been studied 
by many authors and in various domains. It takes the following form. 

\begin{thm}\label{thm:cubature}
Let $\sw$ be a doubling weight on $\Omega$. Let $\Xi$ 
be a maximum $\frac{\delta}{n}$-separated subset of $\Omega$. There is a $\delta_0 > 0$ 
such that for $0 < \delta < \delta_0$ there exist positive numbers $\l_z$, $z \in \Xi$, so that 
\begin{equation}\label{eq:CFgeneral}
    \int_{\Omega} f(x) \sw(x)  \d \sm(x) = \sum_{z \in \Xi }\l_z f(z), \qquad 
            \forall f \in \Pi_n(\Omega).
\end{equation}
Moreover, the weight $\l_z$ satisfies $\l_z \ge c_1 \sw\!\left(B(z, \tfrac{\delta}{n})\right)$. 
\end{thm}

This theorem is first proved for the unit sphere \cite{MNW, NPW1}. The proof is standard by now and 
can be easily adapted to any space of homogeneous type. The existence of positive $\l_z$ in \eqref{eq:CFgeneral}
is based on the Farkas lemma and uses the Marcinkiewicz-Zygmund inequality. The proof latter
inequality does not rely on highly localized kernels, nor is the proof of the above theorem. Thus, 
these results are properties for the space of homogeneous type. With such generality, they have 
become a useful tool for the discretization of $L^p(\Omega,\sw)$ for fairly general $\Omega$ and $\sw$. 

In the case that the space is localizable, we can utilize the highly localized kernels to provide a matching 
upper bound for the coefficients $\l_z$ via the Christoffel function, which plays an important role in 
studying the tight polynomial frame in the next subsection. Let $\sw$ be a doubling weight on $\Omega$. 
The Christoffel function $\l_n(\sw;\cdot)$ is defined by 
\begin{align}\label{eq:ChristoffelF}
   \l_n(\sw;x): = \inf_{\substack{g(x) =1 \\ g \in \Pi_n(\Omega)}} \int_{\Omega} |g(x)|^2 \sw(x)  \d \sm(x).
\end{align} 
The function is of interest in its own right and has various applications. Recall that $P_n(\sw;\cdot,\cdot)$ 
denote the reproducing kernel of $\CV_n(\Omega,\sw)$. it is known \cite[Theorem 3.6.6.]{DX} that, for 
$n=0,1,2,\ldots$,  
\begin{align}\label{eq:ChristoffelF2} 
   \l_n(\sw;x) = \frac{1}{K_n(\sw; x,x)}, \qquad   K_n(\sw;x,y) := \sum_{k=0}^n P_k(\sw; x,y). 
\end{align}    

We need to obtain lower and upper bound for $\l_n(\sw;x)$. If $(\Omega,\sw, \sd)$ is localizable, then the 
lower bound follows immediately from the highly localized kernel $L_n(\sw)$, since 
$$
  K_n(\sw; x,x) \le L_n (\sw; x,x) \le c\, \sw\big( B(x,n^{-1}) \big) 
$$
where the second inequality follows from setting $y=x$ in Assertion 1, which is equivalent to 
$\l_n(\sw; x) \ge c  \sw\big( B(x,n^{-1}) \big)$ by \eqref{eq:ChristoffelF2}. For the upper bound, we can use 
\eqref{eq:ChristoffelF} by selecting a particular function $g(y)$ that is localized at $x$, which we choose as
a fast decaying polynomial. In one variable, fast decaying polynomials are studied in \cite{KT}, see also
 \cite{ST}, which leads to fast decaying polynomials on the unit sphere by the addition formula. For the
unit ball, such polynomials are constructed in \cite{PX2}. In the space of homogeneous type, we assume
that such polynomials exist by making the following assertion. 

\medskip
{\bf Assertion 4}.  {\it Let $\Omega$ be compact. For each $x \in \Omega$, there is a nonnegative 
polynomial $T_x$ of degree at most $n$ that satisfies 
\begin{enumerate}[   (1)]
\item $T_x(x) =1$, $T_x(y) \ge \delta > 0$ for $y \in B(x,\f 1 n)$ for some $\delta$ independent of $n$,
and, for each $\g > 1$,  
$$
     0 \le  T_x(y) \le c_\g (1+ n \sd(x,y))^{-\g}, \qquad y \in \Omega; 
$$
\item there is a polynomial $q_n$ such that $q_n(x) T_x(y)$ is a polynomial of degree at most $r n$,
for some positive integer $r$, in $x$-variable and $c_1 \le q_n(x) \le c_2$ for $x \in \Omega$ for some
positive numbers $c_1$ and $c_2$. 
\end{enumerate}
}
\medskip

This was formulated in \cite{X21} and used to establish the following \cite[Lemma 2.1.5]{X21}:

\begin{lem}\label{lem:Ass4Q}
Assume Assertion 4. Let $\a> 0$ be a positive number and let $x \in \Omega$ be fixed. For a doubling 
weight $\sw$ on $\Omega$, there is a polynomial $Q_x$ of degree $n$ such that, for all $y \in \Omega$, 
\begin{equation} \label{eq:Ass4Q}
    c_1 (1+n \sd(x,y))^{\a} \sw(B(y,\tfrac{1}{n}))\le Q_x(y) \le c_2 (1+n \sd(x,y))^{\a} \sw(B(y,\tfrac{1}{n})), 
\end{equation}
where $c_1$ and $c_2$ are positive constant independent of $n$, $x$ and $y$. 
\end{lem}

The inequality \eqref{eq:Ass4Q} shows that $Q_x$ is highly localized and decays near-exponentially 
away from $y=x$. Using this lemma, we obtain an upper bound of $\l_n(\sw)$ and, in particular, a matching
upper bound for $\l_z$ in Theorem \ref{thm:cubature}.

\begin{thm}\label{thm:cubature2}
Let $(\Omega,\sw,\sd)$ be a localizable space of homogeneous type and assume Assertion 4. Then 
$$
   \l_n(\sw; x) \sim \sw\big(B(y,\tfrac{1}{n})\big).
$$
Moreover, the coefficients $\l_z$ in the cubature rule \eqref{eq:CFgeneral} satisfies 
 $\l_z \sim \sw\!\left(B(z, \tfrac{\delta}{n})\right)$.
\end{thm}

It is shown in \cite{X21, X23a} that this theorem holds for conic domains.  

\subsubsection{Localized polynomial tight frames}
The study of such tight frames starts with the unit sphere in \cite{NPW1}, where they are called {\it needlelits},
and then studied for several regular domains, such as (cf. \cite{BD, DaiX, KPPX, IPX, PX1, PX2, PX3}), and they
have been applied in both computational analysis (cf. \cite{BKMP1, BKMP2, KP1, KP2, LSWW, WLSW}) and in
decomposition of functions spaces (cf. \cite{KPX1, KPX2, IPX2, NPW2, PX3}). We follow the framework
in \cite{X21}. 

Let $(\Omega, \varpi, \sd)$ be a localizable homogeneous space. We consider a highly localized kernel
$L_n(\sw;x,y)$ with its cut-off function $\wh a$ satisfies 
\begin{align*}
\begin{split} 
 \wh a(t)\ge \rho > 0, \qquad & \mbox{if $t \in [3/5, 5/3]$},\\
 [\wh a(t)]^2 + [\wh a(2t)]^2 =1, \qquad & \mbox{if $t \in [1/2, 1]$.}
\end{split}
\end{align*}
Such a cut-off function can be easily constructed. 
The last assumption on $\wh a$ implies 
\begin{equation}\label{eq:a3}
\sum_{j=0}^\infty  \left[ \wh a\left( \frac{t}{2^{j} } \right) \right]^2 = 1, \quad t \in [1, \infty).
\end{equation}
In terms of the kernel $L_n(\sw;\cdot,\cdot)$ defined via the above cut-off function $\wh a$, we define 
$$
  F_0(x,z):=1, \quad \hbox{and}\quad F_j(x,z): = L_{2^{j-1}}(\sw; x, z), \quad j = 1, 2, 3, \ldots. 
$$
and define, accordingly, $F_j * f = L_{2^{j-1}}(\sw)* f$, that is, 
$$
   F_j * f (x): = \int_\Omega f(y) F_j(x,y) \sw(y)  \d \sm(y), \qquad j = 0,1,2,\ldots. 
$$
The $*$ operation is associative. By the orthogonality of the reproducing kernel, it follows readily that
\begin{align*}
F_j*F_j * f = \sum_{k=1}^{2^{j}} \left|\wh a\left (\frac{k}{2^{j-1}}\right)\right|^2 \proj_k(\sw; f).
\end{align*}
Hence, the following semi-discrete Calder\'{o}n type decomposition follows from \eqref{eq:a3}, 
\begin{equation}\label{eq:f=F*F*f}
   f  =  \sum_{k=0}^\infty \proj_k(\sw;f) = \sum_{j=0}^\infty F_j* F_j * f, \qquad f\in L^2(\Omega,\sw). 
\end{equation}

Since $F_j(x,\cdot)F_k(\cdot,y)$ is a polynomial of degree $2^j$, discretization by a cubature rule of degree $2^j$ 
puts $F_j*F_k$ as a sum, which leads to a tight frame if the cubature rule has positive coefficients. We use the
positive cubature rule in Theorem \ref{thm:cubature}. For $j =0,1,\ldots,$ let $\ve_j = \frac \delta {2^{j}}$ and
let $\Xi_j$ be a maximal $\ve_j$-separated subset in $\Omega$, where $\delta$ is chosen so that the cubature rule 
\eqref{eq:CFgeneral} holds. Thus, there are $\l_{z,j} > 0$ for $z \in \Xi_j$ such that 
\begin{equation*}
  \int_\Omega f(x) \sw(x)  \d \sm(x) = \sum_{z \in \Xi_j} \l_{z,j} f(z), \qquad f \in \Pi_{2^j} (\Omega); 
\end{equation*}
moreover, $\l_{z,j} \sim \sw(B(z, 2^{-j}))$ since we assume that Assertion 4 holds for $\sw$. We use this cubature 
to discretize $F_j*F_k$, which leads to a tight frame. 
 
For $z \in \Xi_j$ and $j = 1, 2,\ldots$, we define our frame elements by 
$$ 
      \psi_{z,j}(x):= \sqrt{\l_{z,j}} F_j(x, z). 
$$ 
Then $\Psi:= \big\{\psi_{z,j}: z \in \Xi_j, \quad 1 \le j \le \infty\big\}$ is a frame system \cite[Theorem 2.20]{X21}:

\begin{thm}\label{thm:frame}
Assume that $\Omega$ admits a localizable homogenous space. Let $\sw$ be a doubling weight 
satisfying Assertion 4. If $f\in L^2(\Omega, \sw)$, then
\begin{equation} \label{eq:f=frame}
   f =\sum_{j=0}^\infty \sum_{z \in\Xi_j}
            \langle f, \psi_{z, j} \rangle_\sw \psi_{z,j}  \qquad\mbox{in $L^2(\Omega, \sw)$}
\end{equation}
and
\begin{equation} \label{eq:tight-frame}
\|f\|_{2, \sw}  = \Big(\sum_{j=0}^\infty \sum_{z \in \Xi_j} |\langle f, \psi_{z,j} \rangle_\sw|^2\Big)^{1/2}.
\end{equation}
\end{thm}

The frame element $\psi_{z,j}$ has a near exponential rate of decay away from its center with respect to the 
distance $\sd(\cdot, \cdot)$ on $\Omega$. 

\begin{thm}\label{p:localization}
Let $(\Omega,\varpi, \sd)$ be a localizable space of homogeneous type. There, there is a constant $c_\k >0$ 
depending only on $\k$, $d$, $\varpi$ and $\wh a$ such that for $z \in \Xi_j$, $j=0, 1, \dots$,  
\begin{equation} \label{est.needl}
   |\psi_{z,j}(x)| \le c_\k \frac{1}{\sqrt{\varpi(B(z, 2^{-j}))} (1+ 2^j \sd(x,z))^\k}, \quad x\in \Omega.
\end{equation}
\end{thm}

Besides the regular domains mentioned before, these localized tight frames hold on conic domains as
shown in \cite{X21}. 

\section{Examples of localizable space of homogeneous type}
\setcounter{equation}{0}
We state several examples of $(\Omega,\sw, \d)$ that are localizable spaces of homogeneous type. In each 
case, we also include the addition formula utilized to prove Assertions 1 and 2. 

\subsection{The interval $[-1,1]$ with the Jacobi weight} 
For $\a, \b > -1$, the Jacobi weight function is defined by 
$$
      w_{\a,\b}(t):=(1-t)^\a(1+t)^\b, \qquad -1 < x <1. 
$$
The intrinsic distance of the interval $[-1,1]$ is defined by 
\begin{equation*}\label{eq:dist[-1,1]}  
\sd_{[-1,1]}(t,s) = \arccos \left(t s + \sqrt{1-t^2} \sqrt{1-s^2}\right),
\end{equation*}
which is the projection of the distance $|\t-\phi|$ on the upper half of the unit circle if we set $t  =\cos \t$ 
and $s = \cos \phi$. The space $([-1,1], w_{\a,\b}, \sd_{[-1,1]})$ is a localizable homogeneous space if
$\a, \b \ge -\f12$. Let $P_n^{(\a,\b)}$ be the usual Jacobi polynomials. Then the kernel $L_n$, denoted 
by $L_n^{(\a,\b)}$, is of the form 
\begin{equation}\label{def.L}
L_n^{(\a,\b)}(x,y)=\sum_{j=0}^\infty \wh a \Big(\frac{j}{n}\Big)
       \frac{P_j^{(\a,\b)}(x) P_j^{(\a,\b)}(y)} {h_j^{(\a,\b)}},
\end{equation}
where $\wh a$ is an admissible cutoff function and $h_j^{(\a,\b)}$ is the norm square of $P_j^{(\\a,\b)}$ in
the $L^2([-1,1], w_{\a,\b})$. Let the ball $B(x,r)$ be defined with respect to this distance. Then the Jacobi weight 
is a doubling weight and satisfies, 
$$
   w_{\a, \b}(B(x, n^{-1})) \sim n^{-1} (1-t + n^{-2})^{\a+\f12} (1+t + n^{-2})^{\b+\f12} =: w_{\a,\b}(n; t).
$$
The first Assertion of the highly localized kernel is established in \cite{PX1}, which states: for any $\k >0$ and $n \ge 1$,
\begin{equation*}
|L_n^{(\a,\b)} (t,s)| \le c_\k \frac{1}{\sqrt{w_{\a,\b}(n; t)}\sqrt{w_{\a,\b}(n; s)}}
     \left(1+n \sd_{[-1,1]}(t, s)\right)^{- \k}, 
\end{equation*}
where $\alpha, \beta \ge -1/2$ and $0<\ve \le 1$. The second assertion is stated in \cite[Theorem 2.2]{KPX1}.

In this one-dimensional case, there is no addition formula but a product formula that writes the product of two
Jacobi polynomials of the same parameters as an integral of one Jacobi polynomial of the same parameters,
which allows us to write $L_n^{(\a,\b)}(x,y)$ as an integral of $L_n^{(\a,\b)}(z,1)$.  

\subsection{Unit sphere}\label{sec:sphere}
For the unit sphere $\sph$ of $\RR^d$, we consider the space $(\sph, \d \s, \sd_\SS)$, where $\d\s$ is the 
surface measure and $\sd_\SS$ is the geodesic distance defined by
$$
  \sd_\SS (\xi,\eta) = \arccos \la \xi,\eta\ra, \qquad \xi, \eta \in \sph.
$$ 
The orthogonal polynomials are the classical spherical harmonics, which are homogeneous harmonic polynomials restricted 
to the unit sphere. The space $\CV_n^d(\sph, \d\s)$ is usually denote by $\CH_n^d$ and its dimension is given by 
\eqref{eq:dimVnS}. Let $\{Y_{\nu,n}: 1 \le \nu \le \dim \CH_n^d\}$ be an orthonormal basis of $\CH_n^d$. The classical
addition formula for the spherical harmonics gives a closed-form formula for the reproducing kernel of $\CH_n^d$, 
$$
 \sP_n (\xi,\eta):= \sum_{1 \le \nu \le \dim \CH_n^d} Y_{\nu,n}(\xi) Y_{\nu,n}(\eta) =
  Z_n^{\frac{d-2}2}  (\la \xi,\eta\ra), \quad \forall \xi,\eta\in \sph,
$$
where we use $Z_n^\l$ to denote a multiple of the Gagenbauer polynomials $C_n^\l$, 
$$
  Z_n^\l(t) : = \frac{n+\l}{\l} C_n^\l (t), \qquad \l \ge 0,
$$
where the case $\l = 0$ holds under the limit so that $Z_n^0(t) = T_n(t)$, the Chebyshev polynomial of the first kind. 
The addition formula reduces the estimate of the highly localized kernel to the kernel of the Jacobi polynomials 
$P_n^{(\mu,\mu)}(1,z)$ with $\mu = \frac{d-3}{2}$. In this case, the Assertion 1 is of a particularly simple form
$$
   \left| L_n(\xi,\eta) \right | \le c_\k \frac{n^{d-1}}{(1+ n \sd_\SS (\xi,\eta))^{\k}},
$$
which is also the first highly localized kernel known in the literature (\cite{NPW1}).

The space $\CH_n^d$ has another important property worth mentioning: it is the eigenspace of the Laplace-Beltrami 
operator $\Delta_0$, which is the restriction of the Laplace operator $\Delta$ on the unit sphere; see, for example,
\cite[Section 1.4]{DaiX} for its explicit expression. More precisely, it is known \cite[(1.4.9)]{DaiX} that 
\begin{equation} \label{eq:sph-harmonics}
     \Delta_0 Y = -n(n+d-2) Y, \qquad Y \in \CH_n^d.
\end{equation}
We shall call $\Delta_0$ the spectral operator. The two properties, the addition formula and the spectral operator, are 
fundamental for approximation and harmonic analysis on the unit sphere. 

\subsection{Unit ball}
For the unit ball $\BB^d$ of $\RR^d$, we consider the space $(\BB^d, \sw_\mu, \sd_\BB)$, where $\sw_\mu$ is the 
weight function defined by 
$$
  \sw_\mu(x) = (1- \|x\|^2)^{\mu-\f12}, \qquad \mu > -\f12, \quad x \in \BB^d,
$$ 
and $\sd_\BB$ is the intrinsic distance on $\BB^d$ defined by 
$$
  \sd_\BB(x,y) = \arccos \Big(\la x, y \ra + u \sqrt{1-\|x\|^2}\sqrt{1-\|y\|^2} \Big), \quad x,y \in \BB^d. 
$$
Several orthogonal basis of the space $\CV_n^d(\BB^d, \sw_\mu)$ can be given explicitly. The 
reproducing kernel $P_n(\sw_\mu; \cdot,\cdot)$ satisfies an addition formula \cite{X99}: 
for $\mu \ge 0$ and $x,y \in \BB^d$,
\begin{align*}
      P_n(\sw_\mu;x,y) = c_{\mu-\f12} 
         \int_{-1}^1 Z_n^{\mu+\f{d-1}{2}} & \Big(\la x, y \ra + u \sqrt{1-\|x\|^2}\sqrt{1-\|y\|^2} \Big) (1-u^2)^{\mu-1}\d u, \notag
\end{align*}
where $\mu > 0$ and it holds for $\mu =0$ under the limit
\begin{equation}\label{eq:limitInt}
   \lim_{\mu \to 0+}  c_{\mu-\f12} \int_{-1}^1 f(t) (1-t^2)^{\mu-1} \d u = \frac{f(1) + f(-1)}{2}. 
\end{equation}

It is known that $\sw_\mu$ is a doubling weight on the unit ball and 
$$
   \sw_\mu\big( B(x,n^{-1})\big) \sim n^{-d} \left(\sqrt{1-\|x\|^2} + n^{-1} \right)^{2\mu} =:\sw_\mu(n;x).
$$
In this case, the Assertion 1 takes the form 
$$
  \left| L_n (\sw_\mu; x,y)\right| \le c_k \frac{1}{\sqrt{\sw_\mu(n;x)} \sqrt{\sw_\mu(n,y)} (1+n \sd_\BB(x,u))^\k}.
$$ 
This inequality and Assertion 2 are established in \cite{PX2}. Consequently, the space $(\BB^d, \sw, \sd)$ is a 
localizable space of homogeneous type.

\subsection{Simplex} 
For the simplex $\triangle^d$ of $\RR^d$ defined by
$$
  \triangle^d := \left\{x: x_1,\ge 0, \ldots, x_d \ge 0, \, |x| \le 1\right\}, \qquad |x|: = x_1+\ldots + x_d,
$$
we consider the space $(\triangle, \sw_\g, \sd_\triangle)$, where $\sw_\g$ is the classical Jacobi weight function 
$$
   \sw_\g(x) = x_1^{\g_1} \cdots x_d^{\g_d} (1-|x|)^{\g_{d+1}}, \qquad \g_i > -1, \quad x \in \triangle^d,
$$
and $\sd_\triangle$ is the intrinsic distance on $\triangle^d$ defined by 
$$
  \sd_\triangle (x,y) = \arccos \Big(\sqrt{x_1 y_1} + \cdots + \sqrt{x_d y_d} + \sqrt{1-|x|}\sqrt{1-|y|} \Big). 
$$
Like the case of the unit ball, several orthogonal bases of the space $\CV_n^d(\triangle^d, \sw_\g)$ can be given explicitly. 
The reproducing kernel $\sP_n(\sw_\g; \cdot,\cdot)$ also satisfies an addition formula \cite[Theorem 5.2.4]{DX}:  
for $\a_i \ge -\f12$ and $x,y \in \triangle^d$,  
\begin{equation}\label{eq:add-formula}
   P_n(\sw_\g; x,y) = c_\a \int_{-1}^1 Z_{n}^{(|\g| + d-\f12, -\f12)} \left(2 \xi(x,y; t)^2 -1\right) \prod_{i=1}^{d+1} (1-t_i^2)^{\g_i - \f12} \d t,
\end{equation}
where $Z_n^{(\a,\b)}$ is a constant multiple of the Jacobi polynomial
$$
  Z_n^{(\a,\b)}(t) = \frac{P_n^{(\a,\b)}(1)P_n^{(\a,\b)}(t)}{h_n^{(\a,\b)}}
$$
and $\xi(x,y,t)$ is a function defined by 
$$
  \xi(x,y; t) = \sqrt{x_1 y_1} t_1 + \cdots +\sqrt{x_d y_d} t_d + \sqrt{x_{d+1} y_{d+1}} t_{d+1}.
$$
In this case, $\sw_\g$ is known to be a doubling weight and it satisfies, for $\g_i \ge 0$, $1\le i \le d+1$, 
$$
    \sw_\g (B(x,n^{-1})) \sim n^{-d}  \prod_{i=1}^{d+1} (\sqrt{x_i}+n^{-1})^{2 \g_i+1}, \quad x\in \triangle^d,
$$
which follows, for example, from \cite[Lemma 11.3.6]{DaiX} by making a change of variables $x_j \mapsto \sqrt{x_j}$. 
In this case, Assertion 1 was established in \cite{IPX} and Assertion 2 appeared more recently in \cite{GX} for 
$\sw_\g$. Thus, the space $(\triangle^d, \sw, \sd_\triangle)$ is a localizable space of homogeneous type.  

\subsection{Conic surface} 
For the conic surface $\VV_0^{d+1}$ of $\RR^{d+1}$ defined by
$$
  \VV_0^{d+1} := \left\{(x,t): \|x\| = t, \, x \in \RR^d \, 0 \le t \le 1\right\}
$$
we consider the space $(\VV_0^{d+1}, \sw_{-1,\g}, \sd_{\VV_0})$, where the weight function $\sw_{-1,\g}$ is the
Jacobi weight function with one parameter equal to $-1$, 
$$
  \sw_{-1,\g}(t) = t^{-1} (1-t)^{\g}, \qquad \g > -1,
$$
and $\sd_{\VV_0}$ in the intrinsic distance function on the conic surface \cite[Definition 4.1]{X21},
\begin{equation*}
  \sd_{\VV_0} ((x,t), (y,s)): =  \arccos \left(\sqrt{\frac{\la x,y\ra + t s}{2}} + \sqrt{1-t}\sqrt{1-s}\right).
\end{equation*}
Orthogonal polynomials in this setting were first studied in \cite{X20a}, where it is shown, in particular, that 
the reproducing kernel $P_n(\sw_{-1,\g}; \cdot,\cdot)$ for the space $\CV_n(\VV^{d+1}, \sw_{-1,\g})$ satisfies 
the addition formula 
\begin{align}\label{eq:Ln-intV0}
P_n (\sw_{-1,\g}; (x,t), (y,s) )=  
    c_{\g}  \int_{[-1,1]^2} & Z_n ^{(\g + d -\f32,-\f12)}\big(2 \zeta (x,t,y,s; v)^2-1 \big)\\
  &  \times    (1-v_1^2)^{\f{d-2}2-1}(1-v_2^2)^{\g-\f12} \d v \notag
\end{align}
for $\g > -\f12$ and under the limit when $\g = -\f12$. The weight $\sw_{-1,\g}$ is an integrable doubling weight with 
respect to $\sd_{\VV_0}$ and it satisfies \cite[(4.4)]{X21}
\begin{align*} 
 \sw_{\b,\g}\big(\sc((x,t), n^{-1})\big) \sim  r^d (t+ n^{-2} )^{\f{d-2}{2}} (1-t+ n^{-2})^{\g+\f12}. 
\end{align*}
Both Assertions 1 and 2 are established in \cite{X21} with the help of the addition formula \eqref{eq:Ln-intV0}
for $\sw_{-1,\g}$, where the work of carrying out the estimates is fairly involved. Consequently, the space 
$(\VV_0^{d+1}, \sw, \sd_{\VV_0})$ is a localizable space of homogeneous type.  

\subsection{Solid cone}
For the solid cone $\VV^{d+1}$ of $\RR^{d+1}$ bounded by $\VV_0^{d+1}$ and the hyperplane $t=1$, 
$$
  \VV^{d+1} := \left\{(x,t): \|x\| \le t, \, x \in \RR^d \, 0 \le t \le 1\right\}
$$
we consider the weight function $\sw_{\g,\mu}$ defined by 
$$
  \sw_{\g,\mu}(x,t) = (t^2- \|x\|^2)^{\mu-\f12} (1-t)^{\g}, \qquad \b > -\f12, \quad \g > -1,
$$
and $\sd_{\VV}$ is the intrinsic distance function on the conic surface \cite[Definition 4.1]{X21},
\begin{equation*}
  \sd_{\VV} ((x,t), (y,s)): =  \arccos \left(\sqrt{\frac{\la x,y\ra + \sqrt{1-\|x\|^2}\sqrt{1-\|y\|^2}  + t s}{2}} + \sqrt{1-t}\sqrt{1-s}\right).
\end{equation*}
Orthogonal polynomials in this setting were first studied in \cite{X20a}, where it is shown, in particular, that 
the reproducing kernel $P_n(\sw_{\g,\mu}; \cdot,\cdot)$ for the space $\CV_n(\VV^{d+1}, \sw_{\g,\mu})$ satisfies 
the addition formula 
\begin{align}\label{eq:Ln-intV}
P_n \left(\sw_{\g,\mu}; (x,t), (y,s)\right)=  
    c  & \int_{[-1,1]^3}  Z_n ^{2\mu+\g+d} \big(2 \zeta (x,t,y,s; u, v) \big)\\
  &  \times  (1-u^2)^{\mu-1}  (1-v_1^2)^{\mu+\f{d-3}2}(1-v_2^2)^{\g-\f12} \d u \d v. \notag
\end{align}
where $\zeta(x,t,y,s;u,v) \in [-1,1]$ is defined by 
$$
  \zeta(x,t,y,s;t) = v_1\sqrt{\f12 \left(t s + \la x,y\ra + \sqrt{t^2-\|x\|^2}\sqrt{s^2-\|y\|^2} u\right)} + v_1 \sqrt{1-t}\sqrt{1-s}.
$$
The $\sw_{\g,\mu}$ is an integrable doubling weight with respect to $\sd_{\VV}$ and it satisfies \cite[(4.4)]{X21}
\begin{align*} 
 \sw_{\g,\mu}\big(B((x,t), n^{-1})\big) \sim  n^{-d+1}  (t+ n^{-2} )^{\f{d-1}{2}} (1-t+ n^{-2})^{\g+\f12} (t^2-\|x\|^2+n^{-2})^\mu. 
\end{align*}
Both Assertions 1 and 2 are established in \cite{X21}, so that the space $(\VV^{d+1}, \sw, \sd_{\VV})$ is a 
localizable space of homogeneous type. 

\subsection{Double conic and hyperbolic surface}
For $\varrho \ge 0$, we consider the quadratic surface defined by 
$$
  \XX_0^{d+1} =  \left \{(x,t): \|x\|^2 = t^2 - \varrho^2, \, x \in \RR^d, \, \varrho \le |t| \le \sqrt{\varrho^2 +1}\right\}, 
$$
which is a hyperbolic surface if $\varrho >0$ and a double conic surface if $\varrho =0$. We consider
the space $(\XX_0,\sw_{\b,\g}, \sd_{\XX_0})$, where $\sw_{\b,\g}$ is the weight function defined by,
for $d \ge 2$, $\b > -\f12$ and $\g > -\f12$, 
\begin{equation*}
   \sw_{\b,\g}(t) = 
      |t| (t^2-\varrho^2)^{\b-\f12}(\varrho^2+1 - t^2)^{\g-\f12},  \qquad  \varrho \le |t| \le \sqrt{\varrho^2 +1},
\end{equation*}
and $\sd_{\XX_0}$ is the distance function on $\XX_0^{d+1}$ given by 
\begin{align*}
 \sd_{\XX_0}((x,t), (y,s)) = \arccos \left (\la x,y\ra + \sqrt{1+\varrho^2-t^2}\sqrt{1+\varrho^2-s^2} \right).
\end{align*}
When $\varrho = 0$, or on the double conic surface, the weight function becomes 
$$
    \sw_{\b,\g}(t)= |t|^{2\b}(1-t^2)^{\g-\f12}. 
$$
The weight function $\sw_{\b,\g}$ on the hyperbolic surface is a doubling weight and satisfies 
$$
  \sw_{\b,\g} \big(B((x,t),n^{-1}) \big) = 
       n^{-d} \big(t ^2-\varrho^2 + n^{-2}\big)^{\b} \big(1 -t^2 + \varrho^2+ n^{-2}\big)^{\g}. 
$$

The orthogonal polynomials in this setting are first studied in \cite{X20b}. The space of orthogonal
polynomials $\CV_n(\XX_0^{d+1}, \sw)$ of degree $n$ has the same dimension as the space of 
spherical harmonics $\CH_n^{d+1}$. If the weight function $\sw$ is even in the $t$ variable, this space 
can be factored as 
$$
  \CV_n\left(\XX_0^{d+1}, \sw\right) =  \CV_n^\sE\big(\XX_0^{d+1}, \sw\big) \bigoplus 
      \CV_n^\sO\big(\XX_0^{d+1}, \sw\big),
$$
where the subspace $\CV_n^\sE(\XX_0^{d+1}, \sw)$ consists of orthogonal polynomials that are even in 
the $t$ variable, and the subspace $\CV_n^\sO(\XX_0^{d+1}, \sw)$ consists of orthogonal polynomials 
that are odd in the $t$ variable. Then 
\begin{equation*}
  \dim \CV_n^\sE\big(\XX_0^{d+1},\sw\big) =\binom{n+d-1}{n} \quad \hbox{and} \quad  
    \dim \CV_n^\sO\big(\XX_0^{d+1},\sw\big) =\binom{n+d-2}{n-1}. 
\end{equation*}

It turns out that a closed-form formula holds for $P_n^\sE(\sw_{\b,\g})$ and it has the simplest form when $\b = 0$, in which case
the addition formula is given by  
\cite[(5.11)]{X20b}
\begin{align} \label{eq:sfPEadd0Hyp}
    P_n^\sE \big (\sw_{0,\g}; (x,t),(y,s) \big)& = c_{\g-\f12}  \int_{-1}^1   (1-v^2)^{\g-1}  \\
& \times Z_n^{\g+\frac{d-1}{2}}
 \bigg( \la x,y\ra + v \sqrt{1+\varrho^2-s^2} \sqrt{1+ \varrho^2 -t^2}\bigg) \d v. \notag
\end{align}
The kernel $P_n^\sO(\sw_{\b,\g})$ is related to $P_n^\sE(\sw_\b,\g)$ by a simple relation \cite[(5.3)]{X20b}
\begin{align} \label{eq:sfPOadd}
  P_n^\sO \big (\sw_{\b,\g}; (x,t),(y,s) \big) = \frac{\b+\g+\frac{d+1}{2}}{\b+\f{d+1}2} s t \,P_{n-1}^\sE \big (\sw_{\b+1,\g}; (x,t),(y,s) \big)
\end{align}
for $\b, \g > -\f12$ and $\b \ge - \f12$ under the limit when $\varrho = 0$, which shows that
$P_n^\sO \big (\sw_{-1,\g}; \cdot,\cdot \big)$ has a simple closed formula in the form of 
\eqref{eq:sfPEadd0Hyp} when $\varrho = 0$ but with a different weigh function $\sw_{-1,\g}$ in place of $\sw_{0,\g}$. 

In order to use the addition formula, we need to consider the subspace of homogeneous type $(\XX_0^{d+1}, \sw, \sd)$ 
defined according to the parity. Thus, instead of the kernel $L_n(\sw)$, we consider 
$$
  L_n^\sE\big(\sw; (x,t), (y,s)\big) = \sum_{j=0}^\infty \wh a\left( \frac{j}{n} \right) P_j^\sE \big(\sw; (x,t), (y,s)\big) 
$$
and the kernel $L_n^\sO(\sw)$ defined similarly, which uses only polynomials that are even (odd respectively) 
in the $t$ variable. In this case, we can decompose the surface $\XX_0^{d+1}$ as an upper part and a lower part, 
$$
    \XX_0^{d+1}  = \XX_{0,+}^{d+1} \cup \XX_{0,-}^{d+1} =  \{(x,t) \in \XX_0^{d+1}: t \ge 0\} \cup
      \{(x,t) \in \XX_0^{d+1}: t \le 0\}. 
$$ 
If a function $f(x,t)$ is even in the $t$ variable, then we only need to consider its restriction on $\XX_{0,+}^{d+1}$. 
In particular, if we consider, for example, approximation and tight polynomial frames for functions that are even in the
$t$ variable, then the framework established in \cite{X21} applies. 

The above setup can also be regarded as studying a function $f$ defined on $\XX_{0,+}^{d+1} = \VV_0^{d+1}$ and
extends its definition to $\XX_0^{d+1}$ by defining $f(x,t) = f(x,-t)$. It should be noted, however, that the polynomial 
space $\Pi_n(\VV_0^{d+1})$ of polynomials on $\VV_0^{d+1}$ and the $\Pi_n(\XX_{0,+})^{d+1}$ are not the same
and, in fact, have different dimensions. 

Using $L_n^\sE(\sw)$, both Assertions 1 and 2 are established in \cite{X23a}. Thus, if $\sw$ is a doubling weight 
even in the $t$ variable, then $(\XX_{0,+}^{d+1}, \sw, \sd_{\XX_0})$ is a localizable space of homogeneous type. 

\subsection{Double cone and hyperboloid}
For $\varrho \ge 0$, we consider the domain defined by 
$$
  \XX^{d+1} =  \left \{(x,t): \|x\|^2 \le t^2 - \varrho^2, \, x \in \RR^d, \, \varrho \le |t| \le \sqrt{\varrho^2 +1}\right\}, 
$$
which is the hyperboloid, bounded by the hyperbolic surface $\XX^{d+1}$ and two hyperplanes $t  = \pm 1$, if
$\varrho >0$ and a solid double cone if $\varrho =0$. We consider the space $(\XX^{d+1},\sw_{\b,\g,\mu}, \sd_{\XX})$,
where $\sw_{\b,\g,\mu}$ is the weight function defined by 
\begin{equation*}
   \sw_{\b,\g,\mu}(x,t) = |t| (t^2-\varrho^2)^{\b-\f12} (1+\varrho^2-t^2)^{\g-\f12}(t^2-\varrho^2-\|x\|^2)^{\mu - \f12}
\end{equation*}
and $\sd_{\XX}$ is the distance function on $\XX^{d+1}$ given by  
\begin{align*}
 \sd_{\XX}((x,t), (y,s)) = \arccos \left (\la x,y\ra +\sqrt{1-\|x\|^2}\sqrt{1-\|y\|^2}
    +\sqrt{1+\varrho^2-t^2}\sqrt{1+\varrho^2-s^2} \right).
\end{align*}
The weight function $\sw_{\b,\g,\mu}$ on the hyperboloid is a doubling weight and satisfies 
$$
  \sw_{\b,\g,\mu} \big(B((x,t),n^{-1}) \big) = 
       n^{-d-1}   \big(t^2-\varrho^2 + n^{-2})^\b (1+\varrho^2 - t^2+n^{-2}\big)^{\g} \big(t^2-\varrho^2 -\|x\|^2+n^{-2}\big)^{\mu}.
$$

As in the case of hyperbolic surface, if the weight function $\sw$ is even in the $t$ variable, then the space of 
orthogonal polynomials $\CV_n(\XX^{d+1}, \sw)$ can be factored as
$$
  \CV_n\left(\XX^{d+1}, \sw \right) =  \CV_n^\sE\big(\XX^{d+1}, \sw\big) \bigoplus 
      \CV_n^\sO\big(\XX^{d+1}, \sw\big),
$$
where the subspace $\CV_n^\sE(\XX^{d+1}, \sw)$ consists of orthogonal polynomials that are even in 
the $t$ variable, and the subspace $\CV_n^\sO(\XX^{d+1}, \sw)$ consists of orthogonal polynomials 
that are odd in the $t$ variable. A closed-form formula holds for $P_n^\sE(\sw_{\b, \g,\mu})$, which takes the
simplest form if $\b = \f12$ \cite[Corollary 5.6]{X20b},
\begin{align*} 
    P_n^\sE \big (\sw_{\f12, \g,\mu}; (x,t),(y,s) \big)& = c \int_{-1}^1 \int_{-1}^1 
    Z_n^{\g+\mu+\frac{d}{2}} \big(\zeta(x,t,y,s; u,v)\big) (1-v^2)^{\g-1} (1-u^2)^{\mu-1} \d u \d v,
\end{align*}
where 
\begin{align*} 
\zeta(x,t,y,s; u,v) = \Big( \la x,y\ra \,& +  u \sqrt{t^2 - \varrho^2-\|x\|^2}\sqrt{s^2- \varrho^2 -\|y\|^2} \Big)
    \mathrm{sign}(st)  \\
         & + v \sqrt{1+ \varrho^2 -s^2}\sqrt{1+ \varrho^2- t^2}.
\end{align*}
Moreover, the kernel $P_n^\sO(\sw_{\b,\g,\mu})$ is related to $P_n^\sE(\sw_{\b+1,\g,\mu})$ by a formula 
that is similar to \eqref{eq:sfPOadd}. Notice again the shift in the subscript $\b$ by $1$, which shows that the 
addition formula does not apply to $P_n^\sO \big (\sw_{\f12, \g,\mu})$. 

Defining the kernel $L_n^\sE(\sw)$ that uses only $P_j^\sE(\sw)$ as in the case of hyperbolic surface, which
uses only polynomials that are even in the $t$ variable. Then both Assertions 1 and 2 are established in \cite{X23a}. 
If we decompose the the domain $\XX^{d+1}$ as an upper part and a lower part, 
$$
    \XX^{d+1}  = \XX_{+}^{d+1} \cup \XX_{-}^{d+1} =  \{(x,t) \in \XX^{d+1}: t \ge 0\} \cup
      \{(x,t) \in \XX^{d+1}: t \le 0\},  
$$ 
then $(\XX_{+}^{d+1}, \sw, \sd_{\XX})$ is a localizable space of homogeneous type if $\sw$ is even in the 
$t$ variable. 

\subsection{Other double domains of conic and hyperbolic type}
Orthogonal structure on several solid domains, beyond the domains bounded by quadratic surfaces, is preserved 
under an appropriate structure transformation on the hyperboloid $\XX^{d+1}$, as shown in \cite{X24}. 
This includes the addition formula and highly localized kernels so that we obtain localized space of homogeneous 
type on such domains. We state one example to give a flavor.

For $0<  \mathfrak{b} < 1$, we consider the domain inside the unit ball $\BB^{d+1}$ but outside the ellipsoid 
$\{(x,t): \mathfrak{b}(1-\|x\|^2) \le t^2\} \subset \BB^{d+1}$, which is defined by 
$$
      \mathbb{X}^{d+1} = \left\{(x,t):  \sqrt{1- \tfrac{t^2}{\mathfrak{b}}}\le  \|x\| \le  \sqrt{1- t^2}, \, \, 
          |t| \le 1 \right\}. 
$$
We consider the space $(\XX^{d+1},\sw_{\b,\g,\mu}, \sd_{\XX})$, where the weigh function is defined by 
\begin{align*}
 \sw_{\b,\g}(x,t)  = \big(1-\|x\|^2 - t^2 \big)^\beta  
        \big(t^2- \mathfrak{b} (1-\|x\|^2)\big)^\gamma \big(t^2-\mathfrak{b} + \|x\|^2\big)^{\frac12}
\end{align*}
and the distance function $\sd_XX$ is defined by 
\begin{align*}
 \sd_{\XX}((x,t), (y,s)) = \arccos & \left (\la x,y\ra +\sqrt{1-\|x\|^2}\sqrt{1-\|y\|^2} \right. \\
         & \qquad \left.  +\sqrt{\frac{1-t^2 - \|x\|^2}{1- \mathfrak{b}}}\sqrt{\frac{1-s^2 - \|y\|^2}{1- \mathfrak{b}}} \right).
\end{align*}
Then the corresponding kernel $L_n^\sE(\sw_{\b,\g})$ is highly localized with respect to the distance $\sd_\XX$. 
Consequently, $(\XX_+^{d+1},\sw, \sd_{\XX})$ is a localizable space of homogeneous type if $\sw$ is even in the 
$t$ variable.

\end{document}